%% file: GeneralConvergence.tex
\documentclass[]{article}
\input{preamble}

\def\assumptionautorefname{}

\makeatletter
\renewcommand{\assumptionautorefname}{A\@gobble}
\makeatother

\setlist{nolistsep}

\setenumerate[1]{label=\Roman*.}
\setenumerate[2]{label=\Alph*.}
\setenumerate[3]{label=\roman*.}
\setenumerate[4]{label=\alph*.}

\begin{document}
\date{}
\title{On Convergence of General Truncation-Augmentation Schemes for Approximating Stationary Distributions of Markov Chains}
\author{Alex Infanger\thanks{Institute for Computational \& Mathematical Engineering, Stanford University, USA.} \and Peter W. Glynn\footnotemark[1]\ \thanks{Department of Management Science \& Engineering, Stanford University, USA.} \and Yuanyuan Liu\thanks{School of Mathematics and Statistics,
Central South University, China.}}

\maketitle

\begin{abstract} \noindent
In the analysis of Markov chains and processes, it is sometimes convenient to replace an unbounded state space with a ``truncated'' bounded state space. When such a replacement is made, one often wants to know whether the equilibrium behavior of the truncated chain or process is close to that of the untruncated system. For example, such questions arise naturally when considering numerical methods for computing stationary distributions on unbounded state space. In this paper, we study general truncation-augmentation schemes, in which the substochastic truncated ``northwest corner'' of the transition matrix or kernel is stochasticized (or augmented) arbitrarily. In the presence of a Lyapunov condition involving a coercive function, we show that such schemes are generally convergent in countable state space, provided that the truncation is chosen as a sublevel set of the Lyapunov function. For stochastically monotone Markov chains on $\mathbb Z_+$, we prove that we can always choose the truncation sets to be of the form $\{0,1,...,n\}$. We then provide sufficient conditions for weakly continuous Markov chains under which general truncation-augmentation schemes converge weakly in continuous state space. Finally, we briefly discuss the extension of the theory to continuous time Markov jump processes.\\

\end{abstract}
\section{Introduction}

Let  $X = (X_n: n \ge 0)$ be a positive recurrent Markov chain taking values in an unbounded state space $S$. In many settings, one is interested in producing a positive recurrent Markov chain approximation to $X$ with bounded state space, perhaps motivated by numerical considerations. For example, in analyzing an irreducible positive recurrent Markov chain $X$ on countable state space $S$, numerical computation of the stationary distribution $\pi$ for $X$ requires truncating its state space to a finite subset $A_n$ having $n$ elements. Since the ``northwest corner'' matrix $(P(x,y): x,y \in A_n)$ is necessarily substochastic, it is natural to then ``augment'' the probabilities $P(x,y)$ so as to create a stochastic matrix $P_n = (P_n(x,y): x,y \in A_n)$. One then hopes that the stationary distribution $\pi_n$ of $P_n$ will be a good approximation to $\pi$. However, it is well known that universal convergence (as $n$ tends to $\infty$) for such truncation-augmentation schemes is not guaranteed in general; see, for example, (2.5) in \citet{wolfApproximationInvariantProbability1980} and \autoref{example-1} of this paper.

However, when one constructs the augmentation via a ``fixed state'' augmentation, \citet{wolfApproximationInvariantProbability1980} showed that convergence of $\pi_n$ to $\pi$ is guaranteed when $X$ is an irreducible positive recurrent countable state space Markov chain. In \citet{infangerConvergenceTruncationScheme2022}, we show that such convergence can be generally validated when the underlying Markov chain or process is suitably regenerative. This allows us to develop a convergence theory, in a suitably chosen weighted total variation norm, for irreducible positive recurrent Markov chains on countable state spaces and, more generally, for Harris recurrent Markov chains on general state spaces.

In this paper, we seek conditions under which one can be assured that any truncation-augmentation scheme is convergent. In the countable state space setting, one of us developed such results in \citet{liuAugmentedTruncationApproximations2010}. However, as we shall discuss in Section 2, the argument given there relies on results from \citet{borovkovErgodicityStabilityStochastic1998} that are not correctly stated. Section 2 therefore provides new sufficient conditions guaranteeing universal convergence. Our sufficient conditions require that the truncation set be chosen as a sublevel set of a coercive Lyapunov function. This result suggests that a Lyapunov function for the Markov chain $X$ can be useful in designing convergent truncation-augmentation schemes. We also use our approach to recover a convergence result of \citet{gibsonMonotoneInfiniteStochastic1987} and \citet{tweedieTruncationApproximationsInvariant1998}, for stochastically monotone Markov chains; see \autoref{thm2new}.

In Section 3, we study this convergence question in the setting of Markov chains taking values in a complete separable metric space. We generalize the results of Section 2 by establishing weak convergence of $\pi_n$ to $\pi$ (in contrast to the total variation convergence of \citet{infangerConvergenceTruncationScheme2022}), when $X$ has suitably continuous transition probabilities. The arguments in Sections \ref{sec-2} and \ref{sec-3} do not assume Harris recurrence. Rather, they rely on Prohorov’s theorem and related tightness ideas, and the methods establish convergence for arbitrary augmentations. Section \ref{sec-4} uses an argument based on regeneration to establish that general augmentations are valid for strongly uniformly recurrent Markov chains on general state space, thereby generalizing the known theory for Markov matrices on discrete state space to the continuous setting.

This paper concludes, in Section \ref{sec-5}, with a brief discussion of the related convergence theory for continuous time Markov jump processes.

\section{Truncation-Augmentation for Markov Chains with Countably Infinite State Space}\label{sec-2}
In this section, we consider an irreducible positive recurrent Markov chain $X=(X_n:n\geq 0)$ taking values in a countably infinite state space $S$. We let $P=(P(x,y):x,y\in S)$ be the one-step transition matrix of $X$, and we denote the (unique) stationary distribution of $X$ by $\pi=(\pi(x):x\in S)$. Let $(A_n:n\geq 0)$ be a strictly increasing sequence of subsets of $S$ satisfying $A_0\subset A_1\subset A_2\subset ...$ such that $\bigcap_{n=0}^\infty A_n^c=\emptyset$. For the $n$'th truncation $A_n$, let $B_n=(B_n(x,y):x,y\in A_n)$ be the corresponding ``northwest corner'' truncation of $P$ in which $B_n(x,y)=P(x,y)$ for $x,y\in A_n$. The irreducibility of $P$ then guarantees that there exists at least one row of $B_n$ with a row sum strictly less than 1. 

We say that $P_n$ is a \emph{general augmentation} of $B_n$ if $P_n=(P_n(x,y):x,y\in A_n)$ is a stochastic matrix for which
\begin{align*}
P_n(x,y)\geq B_n(x,y)=P(x,y)
\end{align*}
for $x,y\in A_n$. If there exists a probability distribution $\nu_n=(\nu_n(x):x\in A_n)$ for which
\begin{align}
P_n(x,y) = P(x,y) + \sum_{z\in A_n^c}^{}P(x,z)\nu_n(y),\label{eq-21}
\end{align}
then $P_n$ is said to be a \emph{linear augmentation} of $B_n$. If there exists a probability $\nu=(\nu(x):x\in S)$ for which
\begin{align*}
\nu_n(x) = \frac{\nu(x)}{\sum_{y\in A_n}^{}\nu(y)}
\end{align*}
for $x\in A_n$, then we say that $P_n$ is formed from a \emph{fixed linear augmentation} $\nu$. If $\nu=\delta_y$ for a fixed $y\in S$, where $\delta_y=(\delta_y(x):x\in S)$ is a unit point mass distribution at $y$, then we say that $P_n$ is a \emph{fixed state augmentation} of $B_n$. When $S=\mathbb{Z}_+$ and $A_n=\{0,1,...,n\}$, setting $\nu=\delta_0$ is called \emph{first state augmentation}, whereas the use of $\nu_n=\delta_n$ is called \emph{last state augmentation}.

Let $\Pi_n$ be the set of stationary distributions on $A_n$ associated with $P_n$. Since $|A_n|$ is finite, $\Pi_n$ is always non-empty. It is easily seen that when $P_n$ is a general augmentation, $\Pi_n$ may not be a singleton. However, for linear augmentations (whether fixed or not), $\Pi_n$ is always guaranteed to consist of a single unique stationary distribution $\pi_n$ for $P_n$. This follows because the irreducibility of $P$ guarantees that the probability of an exit to $A_n^c$ is positive from every $x\in A_n$. Once $X$ attempts to exit to $A_n^c$, the chain is re-distributed on $A_n$ according to $\nu_n$; see \eqref{eq-21}. Hence, every state $y\in A_n$ for which $\nu_n(y)>0$ is reachable from every $x\in A_n$, as are all the states reachable from such states $y\in A_n$. Consequently, $P_n$ has exactly one closed communicating class, so that $P_n$ has a unique stationary distribution $\pi_n$; see also p. 261 of \citet{senetaComputingStationaryDistribution1980}.

To make our discussion as self-contained as possible, we now provide an example showing that even when $X$ is very well-behaved, $\pi_n$ may fail to converge to $\pi$ as $n\rightarrow\infty$.

\begin{example}\label{example-1} Suppose that $S=\mathbb{Z}_+$ with $P(2i, 2i+1)=1/2=P(2i,0)$ for $i\in \mathbb Z_+$, and $P(2i+1, 2i+2)=1$ for $i\in \mathbb Z_+$. Then,
\begin{align*}
P^2(2i,0)\geq P(2i,0)P(0,0) =\frac{1}{4},
\end{align*}
for $i\in \mathbb{Z}_+$ whereas,
\begin{align*}
P^2(2i+1,0)\geq P(2i+1, 2i+2)P(2i+2,0)=\frac{1}{2}
\end{align*}
for $i\in \mathbb Z_+$. Hence $P^2(x,0)\geq 1/4$ for $x\in \mathbb Z_+$, so that the two-step transition matrix is a Markov matrix, and $X$ is uniformly ergodic.

Suppose we use last state augmentation. Then, when $A_n=\{0,1,...,n\}$ with $n$ odd, state $n$ is absorbing, and the single closed communicating class corresponding to $P_n$ is just $\{n\}$. It follows that $\pi_n=\delta_n$ for $n$ odd, so that $\pi_n$ fails to converge to the stationary distribution $\pi$ of $X$, despite the fact that $X$ is uniformly ergodic. 
\end{example}
A significant literature has developed over the years, focused on obtaining conditions under which various augmentations are guaranteed to converge. As noted in the Introduction, \citet{wolfApproximationInvariantProbability1980} proved that fixed state augmentation is always convergent. This supplemented earlier work of \citet{golubComputationStationaryDistribution1974}, in which it was shown that last state augmentation converges for upper Hessenberg transition matrices. \citet{gibsonAugmentedTruncationsInfinite1987} showed that general augmentations always converge when $P$ is either a Markov matrix or when it is upper Hessenberg. They also showed that when $P$ is lower Hessenberg and $P_n$ is a linear augmentation for which the sequence $(\nu_n:n\geq 0)$ is chosen to be tight, then $\pi_n$ is guaranteed to converge to $\pi$. \citet{gibsonMonotoneInfiniteStochastic1987} showed that convergence also holds for general augmentations when the underlying Markov chain is stochastically monotone. \citet{tweedieTruncationApproximationsInvariant1998} proved convergence for the special case of last state augmentation when $P$ corresponds to a stochastically monotone Markov chain (although most of the paper, as with much of the subsequent literature, focuses on identifying computable rates of convergence of $\pi_n$ to $\pi$).

In \citet{liuAugmentedTruncationApproximations2010}, convergence of $\pi_n$ to $\pi$ for general augmentations is discussed, in the presence of a Lyapunov condition denoted there as $D1(V,b,C)$. The arguments given there rely on the following result (Theorem 5.1) from \citet{borovkovErgodicityStabilityStochastic1998} (suitably re-stated for our current exposition).\\

\begin{quote} Suppose that $(P_n:n\geq 1)$ is a family of one-step transition matrices defined on a countable state space $S$ for which there exists a finite subset $C\subseteq S$, $\lambda>0$, $c<\infty$, and a probability $\phi=(\phi(x):x\in S)$ such that:
\begin{enumerate}[label=\roman*)]
  \item Under $P_n$, $X$ is guaranteed to hit $C$ from any state $x\in S$;
  \item $P_n(x,y)\geq \lambda\phi(y)$ for $x\in C,y\in S$;
  \item $\max_{x\in C} E_n[\tau(C)|X_0=x]\leq c$, where $E_n(\cdot)$ is the expectation on the path-space of $X$ associated with $P_n$ and $\tau(C)=\inf\{n\geq 1: X_n\in C\}$ is the first return time to $C$.
\end{enumerate}
Then, $P_n$ has a unique stationary distribution $\pi_n$, and
\begin{align}
\sup_{n\geq 1}\sum_{y\in S}^{}|P_n^m(x,y)-\pi_n(y)|\rightarrow 0\label{eq-22}
\end{align}
as $m\rightarrow\infty$.
\end{quote}
\ \\ 

Unfortunately, this result is not valid as stated, as made clear by the following counter-example. (The above statement is also missing the obvious aperiodicity requirement, but the aperiodic example below focuses on the more subtle flaw in the statement.)

\begin{example} Suppose that $S=\mathbb{Z}_+$ with $P_n(i,i-1)=1$ for $i\geq 1$ and $n\geq 1.$ We put $P_n(0,0)=1-(1/(n+1))$ and $P_n(0,n)=1/(n+1)$ for $n\geq 1$. Then, we let $C=\{0\}$, $\lambda=1/2$, and $\phi=\delta_0$. Note that
\begin{align*}
E_n[\tau(C)|X_0=0]=1\cdot \left(1-\frac{1}{n+1}\right)+(n+1)\cdot\frac{1}{n+1}\leq 2,
\end{align*}
so that conditions i), ii), and iii) are all in force, uniformly in $n\geq 1$. For this example,
\begin{align*}
\pi_n(0)=\left(2-\frac{1}{n+1}\right)^{-1}
\end{align*}
with
\begin{align*}
\pi_n(i) = \frac{1}{(n+1)}\left(2-\frac{1}{n+1}\right)^{-1}
\end{align*}
for $1\leq i\leq n$. Also,
\begin{align*}
P_n^m(0,0) \geq \left(1-\frac{1}{n+1}\right)^m
\end{align*}
so 
\begin{align*}
\sup_{n\geq 1}\left[P_n^m(0,0)-\pi_n(0)\right]\geq \sup_{n\geq 1}\left[\left(1-\frac{1}{n+1}\right)^m - \left(2-\frac{1}{n+1}\right)^{-1}\right] = \frac{1}{2},
\end{align*}
so \eqref{eq-22} fails to hold.
\end{example}

In view of this example, we now provide here a new discussion of the convergence theory of general augmentations that does not rely on this mis-stated result from \citet{borovkovErgodicityStabilityStochastic1998}, and that is of interest in its own right. We start by recalling that a function $f:S\rightarrow\mathbb{R}_+$ is said to be \emph{coercive} if the sublevel set $K_m(f)\overset{\Delta}{=}\{x\in S: f(x)\leq m\}$ has finite cardinality for each $m\geq 1$. Our theory relies on the following Lyapunov condition assumption:
\begin{Qtheorem}\label{a1} There exists $g:S\rightarrow\R_+$, $b<\infty$, and a coercive function $r$ such that
\begin{align}
(Pg)(x)\leq g(x)-r(x)+b\label{eq23}
\end{align}
for $x\in S$.
\end{Qtheorem}

\begin{remark} We refer to the function $g$ appearing in \eqref{eq23} as a (stochastic) \emph{Lyapunov function}.
\end{remark}

We note that because $g$ is non-negative, $(Pg)(\cdot)$ is non-negative, so that $r(x)\leq g(x)+b$ for $x\in S$. Hence, if $g(x)\leq m$, it follows that $r(x)\leq m+b$, so $K_m(g)\subseteq K_{m+b}(r)$. Consequently, the sublevel sets of the Lyapunov function $g$ have finite cardinality, so $g$ is also coercive. 

We now assume that we construct our truncation sets based on \autoref{a1}. In particular, we use the Lyapunov function $g$ to design our truncation sequence $(A_n:n\geq 1)$, specifically putting $A_n=K_n(g)$. 
\setcounter{theorem}{0}
\begin{theorem} \label{thm-1} Assume \autoref{a1} and suppose that $A_n=K_n(g)$ for $n\geq 1$. Then, for any general augmentation sequence $(P_n:n\geq 1)$ associated with the $A_n$'s, and for any $\pi_n\in \Pi_n$,
\begin{align}
\sum_{x\in A_n}^{}|\pi_n(x)-\pi(x)|\rightarrow 0\label{eq24}
\end{align}
as $n\rightarrow\infty$.
\end{theorem}

\begin{proof} We first show that the Lyapunov bound \autoref{a1} for $P$ can be extended (uniformly) to the $P_n$'s associated with a sequence of general augmentations. In particular, we observe that for $x\in A_n$,
\begin{align}
\sum_{y\in A_n}P_n(x,y)g(y) &= \sum_{y\in A_n}^{}P(x,y)g(y) + \sum_{y\in A_n}^{}(P_n(x,y)-P(x,y))g(y)\nonumber\\
&\leq \sum_{y\in A_n}^{}P(x,y)g(y) + \sum_{y\in A_n}(P_n(x,y)-P(x,y))\sup_{z\in A_n} g(z)\nonumber\\
&\leq \sum_{y\in A_n}^{} P(x,y)g(y) + \sum_{y\in A_n^c}^{}P(x,y)\inf_{z\in A_n^c}g(z)\nonumber\\
&\leq \sum_{y\in A_n}^{}P(x,y)g(y) + \sum_{y\in A_n^c}^{}P(x,y)g(y)\nonumber\\
&= (Pg)(x)\nonumber\\
&\leq g(x)-r(x)+b,\label{eq25}
\end{align}
which, of course, is the desired uniform version of \eqref{eq23}. (Note that the fact that $A_n$ is chosen as the sublevel set $K_n(g)$ is used in a critical way in our second inequality above.)

We can now apply Corollary 4 of \citet{glynnBoundingStationaryExpectations2008}, p. 202, to \eqref{eq25} to conclude that
\begin{align*}
\sum_{x\in A_n}^{}\pi_n(x)r(x)\leq b
\end{align*}
for $n\geq 1$. For $\epsilon>0$, choose $m=m(\epsilon)$ large enough that $b/m<\epsilon$. Then, Markov's inequality guarantees that
\begin{align*}
\sum_{\substack{x\not\in K_m(r)\\ x\in A_n}}^{}\pi_n(x)\leq b/m<\epsilon
\end{align*}
and hence 
\begin{align*}
\sum_{x\in K_m(r)} \pi_n(x)\geq \sum_{x\in K_m(r)\cap A_n}^{}\pi_n(x)\geq 1-\epsilon
\end{align*}
uniformly in $n\geq 1$. It follows that $(\pi_n:n\geq 1)$ is a tight sequence of probabilities on $S$. Consequently, Prohorov's theorem implies that every subsequence $(\pi_{n_k}:k\geq 1)$ contains a further subsequence $(\pi_{n'_k}:k\geq 1)$ for which
\begin{align}
\pi_{n'_k}(x)\rightarrow\pi'(x)\label{eq26}
\end{align}
at each $x\in S$ as $k\rightarrow\infty$, where $\pi'=(\pi'(x):x\in S)$ is a probability on $S$; see, for example, \citet{billingsleyConvergenceProbabilityMeasures1968}. Then, for $y\in S$,
\begin{align}
\pi'(y) &= \lim_{k\rightarrow\infty}\pi_{n'_k}(y)\nonumber\\
&=\lim_{k\rightarrow\infty} \sum_{x\in A_n}^{}\pi_{n'_k}(y)P_{n'_k}(x,y)\nonumber\\
&\geq \sum_{x\in S}^{}\varliminf_{k\rightarrow\infty} \pi_{n'_k}(x)P_{n'_k}(x,y)\nonumber\\
&= \sum_{x\in S}^{}\pi'(x)P(x,y),\label{eq27}
\end{align}
as a result of Fatou's lemma. Since,
\begin{align}
1 = \sum_{y\in S}^{}\pi'(y) = \sum_{y\in S}\sum_{x\in S}^{}\pi'(x)P(x,y)\label{eq-28}
\end{align}
it is evident that \eqref{eq27} must hold with equality. Since $\pi$ is the unique stationary distribution of $P$, we find that $\pi=\pi'$. Since every convergent subsequence $(\pi_{n'_k}:k\geq 0)$ must have the same limit $\pi$, we conclude that for each $x\in S$, 
\begin{align}
\pi_n(x)\rightarrow\pi(x)\label{eq29}
\end{align}
as $n\rightarrow\infty$. Since $S$ is countably infinite, this easily implies {\eqref{eq24}}. 
\end{proof}

\begin{remark} The novel element in the above proof is establishing tightness. \citet{senetaComputingStationaryDistribution1980} proves that tightness implies convergence; our argument of \eqref{eq26} through \eqref{eq29} is provided in order to make the argument as self-contained as possible. The argument also makes clear that if $A_n\nearrow S$ is chosen so that $\sup_{z\in A_n} g(z)\leq \inf_{z\in A_n^c} g(z)$ for $n\geq 1$, the proof continues to be valid. This observation is used in the proof of \autoref{thm2new}.\label{remark2new}
\end{remark}

\begin{remark}\label{remark2}
For \autoref{example-1}, we note that if $g(2i)=e^{\theta i}$ and $g(2i+1)=e^{\theta(i+3/2)}$ for $i\geq 0$, then 
\begin{align*}
(Pg)(2i) &= \frac{1}{2}g(0) + \frac{1}{2}g(2i+1)\\
&= g(2i) - (1-\frac{1}{2}e^{3\theta/2})e^{\theta i} + \frac{1}{2},
\end{align*}
while 
\begin{align*}
(Pg)(2i+1) &= g(2i+2)\\
&= g(2i+1) - \left(e^{3\theta/2} - e^{\theta} \right)e^{\theta i}.
\end{align*}
Hence, if we choose $\theta>0$ so that $e^{3\theta/2}<2$, we find that \autoref{a1} holds with $r$ coercive. In this case, we note that the sublevel sets of $g$ can only take the form $\{0,1,...,2i\}$ for $i\in \mathbb Z_+$, so that the bad $A_n$'s of \autoref{example-1} are precluded by our choice of Lyapunov function.
\end{remark}

\begin{remark} Note that Theorem \ref{thm-1} establishes that no strictly increasing Lyapunov function $g$ can be constructed for Example \ref{example-1}, for otherwise we would be guaranteed convergence along the sequence of truncation sets given by $A_n=\{0,1,...,n\}$.
\end{remark}

\begin{remark} For \autoref{example-1}, the uniform ergodicity of $X$ implies that one can find a bounded non-negative $g$ for which
\begin{align}
(Pg)(x)\leq g(x)-1\label{eq-210}
\end{align}
for $x\geq 1$. (In particular, we can take $g(x)$ as the expected hitting time of $\{0\}$, starting from $x$.) Inequality \eqref{eq-210} is closely related to condition $D1(V,b,C)$ of \citet{liuAugmentedTruncationApproximations2010}. Note that the sublevel sets of this $g$ form truncation sets that either equal $S$ (when the level is chosen higher than the bound) or take the same form as in \autoref{remark2}. Thus, in this example, \autoref{thm-1} holds, even though the function $r$ associated with \eqref{eq-210} is non-coercive. In particular, we do not know whether \autoref{thm-1} continues to hold when one chooses $A_n=K_n(g)$ in the presence of \eqref{eq-210} holding outside some finite subset $C\subseteq S$. (Of course, we would want to add the requirement that $g$ be coercive, in order to ensure that $K_n(g)$ has finite cardinality for each $n\geq 1$.)
\end{remark}

\begin{remark} When \autoref{a1} holds, the coerciveness of $r$ guarantees that 
\begin{align}
(Pg)(x)\leq g(x)-1\label{eq211}
\end{align}
for $x$ outside some finite subset $C$ (say). This implies that the expected hitting time of $C$ from $x\in C^c$ is bounded above by $g(x)$; see, for example, p. 344 in \citet{meynMarkovChainsStochastic2009}. Hence, \autoref{a1} implies that $C$ can be reached with positive probability from each $x\in C^c$. On the other hand, the irreducibility of $P$ and finiteness of $C$ ensure that each $y\in C$ can be reached from each $x\in C$ by a path lying entirely within $K_n(g)$ when $n$ is chosen sufficiently large. Hence \autoref{a1} guarantees that for $n$ large, $\Pi_n$ is guaranteed to have a single (unique) element.
\end{remark}

We conclude this section by showing that the general augmentation result for stochastically monotone Markov chains due to \citet{gibsonMonotoneInfiniteStochastic1987} is a special case of \autoref{thm-1}, as is the last state augmentation result of \citet{tweedieTruncationApproximationsInvariant1998}. We recall that a Markov chain is \emph{stochastically monotone} on $\mathbb Z_+$ if $\sum_{w\geq y}P(x,w)$ is non-decreasing in $x$ for each $y$.

\begin{theorem}\label{thm2new} Suppose that $X$ is an irreducible and positive recurrent stochastically monotone Markov chain on $\mathbb{Z}_+$ with $A_n$ chosen to be $A_n=\{0,1,...,n\}$. Then, for any general augmentation sequence $(P_n:n\geq 1)$ associated with the $A_n$'s, and for any $\pi_n\in \Pi_n$, 
\begin{align*}
\sum_{x}^{}|\pi_n(x)-\pi(x)|\rightarrow 0
\end{align*}
as $n\rightarrow\infty$, where $\pi=(\pi(x):x\in S)$ is the unique stationary distribution of $X$.
\end{theorem}
\begin{proof} Put
\begin{align*}
\bar{P}(x)=\sum_{y\geq x}^{}\pi(y)
\end{align*}
and set
\begin{align*}
r(x)=\bar{P}(x)^{-\frac{1}{2}}
\end{align*}
for $x\geq 0$. Clearly, $r$ is non-negative, non-decreasing, and coercive. Furthermore, 
\begin{align}
0\leq \sum_{k=0}^{n}r(k)\pi(k) &= \sum_{k=0}^{n}\bar P(k)^{-1/2}(\bar P(k)-\bar P(k+1))\nonumber\\
&= \sum_{k=0}^{n}\int_{\bar P(k+1)}^{\bar P(k)}\bar P(k)^{-1/2}du\nonumber\\
&\leq \sum_{k=0}^{n}\int_{\bar P(k+1)}^{\bar P(k)}u^{-1/2}du\nonumber\\
&=\int_{\bar P(n+1)}^{1}u^{-1/2}du \leq \int_{0}^{1}u^{-1/2}du=2<\infty.\label{eq212-pir-finite}
\end{align}
Hence, \eqref{eq212-pir-finite} shows that
\begin{align*}
\alpha\overset{\Delta}{=} \sum_{k=0}^{\infty}r(k)\pi(k)<\infty.
\end{align*}
As a consequence, we can put $r_c(x)=r(x)-\alpha$ and consider Poisson's equation
\begin{align}
(P-I)\tilde g = -r_c.\label{eq212new}
\end{align}

It is shown in \citet{glynnSolutionsPoissonEquation2022} that because $X$ is stochastically monotone and $r_c$ is non-decreasing, \eqref{eq212new} has a finite-valued non-decreasing solution $\tilde g$. Also, we note that $g(x)=\tilde g(x)-\tilde g(0)$ is guaranteed to be non-negative. Hence,
\begin{align*}
(Pg)(x)= g(x)-r(x)+\alpha
\end{align*}
for $x\geq 0$, where $r$ is non-negative and coercive and $g$ is non-negative. We can therefore apply Theorem \ref{thm-1}. Because $g$ is non-decreasing, the sublevel sets $K_n(g)$ take the form $\{0,1,2,...,g^{-1}(n)\}=A_{g^{-1}(n)}$, so that \autoref{thm-1} and \autoref{remark2new} yield the desired result.
\end{proof}

The key special feature of a stochastically monotone Markov chain is that we may always choose $A_n$ to be an ``interval'' of the form $\{0,1,...,n\}$, and yet retain convergence for general augmentations.

\section{Convergence of General Augmentations for Continuous State Space Markov Chains}\label{sec-3}

In this section, we assume that $X=(X_n:n\geq 0)$ is a Markov chain taking values in a complete separable metric space $X$. For $x,y\in S$, let
\begin{align*}
P(x,dy) = P(X_{n+1}\in dy|X_n=x)
\end{align*}
and let $P=(P(x,dy):x,y\in S)$ be the one-step transition kernel of $X$. We assume that $X$ has a unique stationary distribution $\pi=(\pi(dx):x\in S)$ satisfying the equation 
\begin{align*}
\pi(dy) = \int_{S}^{}\pi(dx)P(x,dy)
\end{align*}
for $y\in S$.

As in \autoref{sec-2}, our theory requires a Lyapunov function assumption. In this setting, we say that a function $f:S\rightarrow\mathbb R_+$ is \emph{coercive} if for each $n\geq 1$, $K_n(f)=\{x\in S:f(x)\leq n\}$ is either empty or compact in $S$.

\begin{Qtheorem}\label{a2} There exists a coercive function $r$, a non-negative function $g:S\rightarrow\R_+$, and $b<\infty$ for which
\begin{align}
\int_{S}^{}P(x,dy)g(y)\leq g(x)-r(x)+b\label{eq-31}
\end{align}
for $x\in S$.
\end{Qtheorem}

As in \autoref{sec-2}, $K_m(g)\subseteq K_{m+b}(r)$. If $g$ is continuous, the fact that $[0,m]$ is closed in $\R_+$ implies that $K_m(g)$ is closed. Also, $K_m(g)$ is a closed subset of the compact set $K_{m+b}(r)$, and so $K_m(g)$ is then compact; see p. 73 of \citet{copsonMetricSpaces1968}. In other words, if $g$ is continuous, \autoref{a2} guarantees that $g$ is coercive.

Let $A_n = K_n(g)$, and put $B_n=(B_n(x,dy): x,y\in A_n)$, where $B_n(x,dy)=P(x,dy)$ for $x,y\in A_n$. Let $P_n=(P_n(x,dy):x,y\in A_n)$ be a general augmentation for $B_n$, so that it is a stochastic kernel on $A_n$ for which 
\begin{align*}
P_n(x,dy) \geq B_n(x,dy) = P(x,dy)
\end{align*}
for $x,y\in A_n$. Suppose that $\Pi_n$ is the set of stationary distributions for $P_n$. Given our assumptions up to this point, $\Pi_n$ may be either empty, consist of a singleton, or contain a multiplicity of elements.

\begin{theorem} \label{thm-2}Suppose \autoref{a2} holds and that $\Pi_n$ is non-empty for $n\geq 1$. If $\pi_n\in \Pi_n$, then $(\pi_n:n\geq 1)$ is a tight sequence of probabilities on $S$.
\end{theorem}
\begin{proof} Because $A_n$ is chosen as a sublevel set of $g$, the same argument as used in \autoref{thm-1} proves that
\begin{align*}
\int_{A_n}^{}P_n(x,dy) g(y) \leq g(x)-r(x)+b
\end{align*}
for $x\in A_n$. Again, Corollary 4 of \citet{glynnBoundingStationaryExpectations2008} proves that
\begin{align*}
\int_{A_n}^{}\pi_n(dx)r(x)\leq b
\end{align*}
for $n\geq 1$. For $\epsilon>0$, choose $m=m(\epsilon)$ so that $b/m<\epsilon$. Then, Markov's inequality implies that $\pi_n(K_m(r)^c)<\epsilon$, and hence
\begin{align*}
\pi_n(K_m(r))\geq 1-\epsilon
\end{align*}
uniformly in $n\geq 1$. Since $K_m(r)$ is compact, this establishes the tightness of $(\pi_n:n\geq 1)$.
\end{proof}

\begin{remark} As in the proof of \autoref{thm2new}, it suffices that $A_n\nearrow S$ with\\ $\sup_{z\in A_n} g(z)\leq \inf_{z\in A_n^c} g(z)$, in order that the result be valid.\\
\end{remark}

We now wish to argue that any weak limit point of $(\pi_n:n\geq 1)$ must equal $\pi$. For this purpose, we let $bC$ be the space of real-valued bounded continuous functions with domain $S$, and assume:

\begin{Qtheorem}\label{a3} If $f\in bC$, then
\begin{align*}
(Pf)(\cdot)\overset{\Delta}{=}\int_{S}^{}f(y)P(\cdot,dy)
\end{align*}
also lies in $bC$.
\end{Qtheorem}

A transition kernel satisfying \autoref{a3} is said to be \emph{weakly continuous} (or equivalently, \emph{Feller continuous}).

\begin{theorem}\label{thm4} Suppose \autoref{a2} and \autoref{a3} hold with $g$ continuous, and that $\Pi_n$ is non-empty for $n\geq 1$. If $\pi_n\in\Pi_n$, then
\begin{align*}
\pi_n\Rightarrow \pi
\end{align*}
as $n\rightarrow\infty$, where $\Rightarrow$ denotes weak convergence in $S$.
\end{theorem}

\begin{proof} In view of \autoref{thm-2}, let $(\pi_{n'_k}:k\geq 1)$ be a weakly convergent subsequence of $(\pi_n:n\geq 1)$, so that there exists a probability $\pi'$ on $S$ for which
\begin{align}
\pi_{n'_k}\Rightarrow\pi'\label{eq-32}
\end{align}
as $k\rightarrow\infty$. For $f\in bC$, it follows that
\begin{align}
\int_{A_{n_k'}}^{}\pi_{n_k'}(dx)f(x) = \int_{S}^{}\pi_{n_k'}(dx)f(x)\rightarrow \int_{S}^{}\pi'(dx)f(x)\label{eq33}
\end{align}
as $k\rightarrow\infty$. Since $\pi_{n_k'}\in\Pi_{n_k'}$,
\begin{align}
\int_{S}^{}\pi_{n_k'}(dx)(P_{n'_k}f)(x) = \int_{S}^{}\pi_{n_k'}(dx)f(x).\label{eq-34}
\end{align}

For $f\in bC$, let $\norm{f}_{}= \sup\{|f(x)|:x\in S\}$. Also, for $\epsilon>0$, the tightness of the sequence $(\pi_n:n\geq 1)$ guarantees the existence of a compact set $K=K(\epsilon)$ for which $\pi_n(K)\geq 1-\epsilon$ uniformly in $n\geq 1$. We can then write
\begin{align}
\biggr|\int_{S}^{}\pi_n&(dx)(P_nf)(x)-\int_{S}^{}\pi_n(dx)(Pf)(x)\biggr|\nonumber\\
&\leq \left|\int_{K}^{}\pi_n(dx)\left((P_nf)(x)-(Pf)(x)\right)\right|+ 2 \epsilon\norm{f}_{}\nonumber\\
&=\left|\int_{K\cap A_n}^{}\pi_n(dx)\left[\int_{A_n}^{}f(y)(P_n(x,dy)-P(x,dy))-\int_{A_n^c}^{}f(y)P(x,dy)\right]\right|\  + \nonumber\\
&\qquad\qquad\qquad\qquad\qquad\qquad\qquad  2\epsilon\norm{f}_{}\nonumber\\
&\leq \norm{f}_{}\int_{K\cap A_n}^{}\pi_n(dx)|P_n(x,A_n)-P(x,A_n)| \ + \nonumber\\
&\qquad\qquad\qquad\qquad\qquad  \int_{K\cap A_n}^{}\pi_n(dx)\int_{A_n^c}^{}|f(y)|P(x,dy) + 2\epsilon\norm{f}_{}\nonumber\\
&\leq \norm{f}_{}\int_{K\cap A_n}^{}\pi_n(dx)|P_n(x,A_n^c)-P(x,A_n^c)|\ + \nonumber \\
&\qquad\qquad\qquad\qquad\qquad  \norm{f}_{}\int_{K\cap A_n}^{}\pi_n(dx)P(x,A_n^c) + 2\epsilon\norm{f}_{}\nonumber\\
&\leq 2\norm{f}_{}\int_{K}^{}\pi_n(dx)P(x,A_n^c) + 2\epsilon\norm{f}_{}.\label{eq35}
\end{align}
We claim that
\begin{align}
\sup_{x\in K}P(x,A_n^c)\rightarrow 0\label{eq36}
\end{align}
as $n\rightarrow\infty$. To prove this, suppose that \eqref{eq36} does not hold. Then, there exists $\delta>0$ and a sequence $(x_n:n\geq 1)$ such that $x_n\in K$ and
\begin{align}
P(x_n, A_n^c)\geq \delta\label{eq37}
\end{align}
for $n\geq 1$. If $P_x(\cdot)\overset{\Delta}{=}P(\cdot|X_0=x)$, \eqref{eq37} is equivalent to requiring that
\begin{align*}
P_{x_n}(X_1\in A_n^c)\geq \delta
\end{align*}
or, in other words,
\begin{align}
P_{x_n}(g(X_1)>n)\geq \delta\label{eq38}
\end{align}
for $n\geq 1$. Because $K$ is compact, we can extract a subsequence $(x_{n_k}:k\geq 1)$ of $(x_n:n\geq 1)$ and $x_\infty\in K$ for which $x_{n_k}\rightarrow x_\infty$ and \eqref{eq38} is in force along the subsequence.

Choose $r\in \mathbb Z_+$ large enough that
\begin{align}
P_{x_\infty}(g(X_1)<r)\geq 1-\frac{\delta}{2}.\label{eq39}
\end{align}
For $h\in bC$, the continuity of $Ph$ (due to \autoref{a3}) implies that $(Ph)(x_{n_k})\rightarrow (Ph)(x_\infty)$ as $k\rightarrow\infty$. It follows that
\begin{align}
P_{x_{n_k}}(X_1\in \cdot)\Rightarrow P_{x_\infty}(X_1\in \cdot)\label{eq310}
\end{align}
as $k\rightarrow\infty$. The continuity of $g$ implies that $\{x:g(x)<r\}$ is open. The weak convergence statement \eqref{eq310} implies that
\begin{align}
\varliminf_{k\rightarrow\infty} P_{x_{n_k}}(g(X_1)<r) \geq P_{x_\infty}(g(X_1)<r)\geq 1-\frac{\delta}{2}.\label{eq311}
\end{align}
Hence, for $n_k>r$, 
\begin{align}
P_{x_{n_k}}(X_1\in A_{n_k}^c) &= P_{x_{n_k}}(g(X_1)>n_k)\nonumber\\
&\leq P_{x_{n_k}}(g(X_1)\geq r) \nonumber\\
&= 1-P_{x_{n_k}}(g(X_1)<r).\label{312}
\end{align}
As a consequence of \eqref{eq311} and \eqref{312}, we find that 
\begin{align*}
\varlimsup_{k\rightarrow\infty} P_{x_{n_k}}(X_1\in A_{n_k}^c)&\leq 1-\varliminf_{k\rightarrow\infty}P_{x_{n_k}}(g(X_1)<r)\\
&\leq 1-(1-\frac{\delta}{2})\leq \frac{\delta}{2},
\end{align*}
contradicting \eqref{eq38} and proving \eqref{eq36}.

With \eqref{eq36} in hand, we find from \eqref{eq35} that
\begin{align*}
\varlimsup_{n\rightarrow\infty} \left|\int_{S}^{}\pi_n(dx)(P_nf)(x) - \int_{S}^{}\pi_n(dx)(Pf)(x)\right|&\leq 2\epsilon\norm{f}_{}
\end{align*}
whenever $f\in bC$. Since $\epsilon$ was arbitrary, we conclude that
\begin{align}
\int_{S}^{}\pi_n(dx)(P_nf)(x)-\int_{S}^{}\pi_n(dx)(Pf)(x)\rightarrow 0\label{eq313}
\end{align}
as $n\rightarrow\infty$. Because of \autoref{a3}, $Pf\in bC$, so \eqref{eq-32} implies that
\begin{align}
\int_{S}^{}\pi_{n_k'}(dx)(Pf)(x)\rightarrow \int_{S}^{}\pi'(dx)(Pf)(x)\label{eq314}
\end{align}
as $k\rightarrow\infty$. In view of \eqref{eq33}, \eqref{eq-34}, \eqref{eq313}, and \eqref{eq314}, we conclude that 
\begin{align*}
\int_{S}^{}\pi'(dy)f(y) = \int_{S}^{}\int_{S}^{}\pi'(dx)P(x,dy)f(y)
\end{align*}
for each $f\in bC$, which implies that 
\begin{align*}
\pi'(dy) = \int_{S}^{}\pi'(dx)P(x,dy).
\end{align*}
So, $\pi'$ is therefore a stationary distribution of $P$, and must coincide with $\pi$ (due to the assumed uniqueness of $\pi$). So, $(\pi_n:n\geq 1)$ can only have one (weak) limit point, namely $\pi$, proving the theorem.
\end{proof}

The question of when $\Pi_n$ is non-empty (at least for $n$ sufficiently large) must be settled separately. One approach is to impose sufficient conditions on the augmentation $P_n$ so as to guarantee that the Markov chain having the one-step transition kernel $P_n$ is positive Harris recurrent on $A_n$.

The other obvious approach is to leverage ideas related to \autoref{a3}. In particular, it is known that if $P_n$ is weakly continuous as a one-step transition kernel on $A_n$, then the compactness of the state space $A_n$ guarantees that $\Pi_n$ is non-empty; see \citet{karrWeakConvergenceSequence1975}. However, if $P_n$ fails to be weakly continuous, then $\Pi_n$ may be empty, even when the augmentation is chosen to preserve as much continuity as possible.

\begin{example} Suppose that $S=[0,2]$, $A=[0,1]$, and
\begin{align*}
P(x,dy) = \begin{cases} \frac{1}{2} \delta_{1+x}(dy) + \frac{1}{2}\delta_2(dy), & 0\leq x\leq \frac{1}{2},\\
\frac{1}{2} \delta_{2-x}(dy) + \frac{1}{2}\delta_2(dy), & \frac{1}{2}\leq x\leq 2,
\end{cases}
\end{align*}
where $\delta_z(\cdot)$ is a unit point mass distribution at $z\in S$. Then, $P$ is weakly continuous on $S$. 
\end{example}

Of course,
\begin{align*}
P_x(X_1\not\in A) = \begin{cases} \frac{1}{2}, & x\in \{0,1\}\\
1, &  0<x<1
\end{cases}
\end{align*}
for $x\in A$. Suppose that our augmentation takes the form
\begin{align}
\tilde P(x,dy) = P(x,dy) + P_x(X_1\not\in A) G(x,dy)\label{eq315}
\end{align}
for $x,y\in A$, where $G=(G(x,dy):x,y\in A)$ is chosen to be weakly continuous. We claim that even when $G$ is so chosen, $\tilde P$ may fail to have a stationary distribution (despite the compactness of its state space). In particular, choose
\begin{align*}
G(x,dy) = \delta_{\frac{x}{2}}(dy)
\end{align*}
for $x,y\in A$. Then,
\begin{align*}
\tilde P(x,dy) &= \begin{cases} \frac{1}{2} \delta_1(dy) + \frac{1}{2}\delta_0(dy), & x=0;\\
\delta_{\frac{x}{2}}(dy), & 0<x<1;\\
\frac{1}{2}\delta_1(dy) + \frac{1}{2}\delta_{\frac{1}{2}}(dy),& x=1.
\end{cases}
\end{align*}

At any initial point $x\in A$, the Markov chain clearly converges weakly to $\delta_0$. As a consequence, for each $f\in bC$, $x\in A$, 
\begin{align*}
\int_{A}^{}f(y)\tilde P^n(x,dy)\rightarrow \int_{A}^{}f(y)\delta_0(dy) = f(0)
\end{align*}
as $n\rightarrow\infty$. Hence, if $\tilde P$ has a stationary distribution $\tilde \pi$, then for $f\in bC$, 
\begin{align*}
\int_{}^{}\tilde\pi (dy) f(y) = \int_{A}^{}\tilde \pi(dx)\int_{A}^{}\tilde P^n(x,dy) f(y)\rightarrow f(0)
\end{align*}
(by the Bounded Convergence Theorem). So, the only possible choice for $\tilde \pi$ is $\delta_0$. But
\begin{align*}
\int_{A}^{}\delta_0(dx) \tilde P(x,dy) = \tilde P(0, dy) \ne \delta_0,
\end{align*}
so $\delta_0$ is not a stationary distribution for $\tilde P$. This example proves that even when augmentations are constructed via \eqref{eq315}, the augmentation may fail to have a stationary distribution.\\

Hence, if one is relying on the compactness of $A_n$ to directly imply that $\Pi_n\neq \emptyset$ (rather than to use, for example, Harris recurrence), one must exercise care in ensuring that $P_n$ is weakly continuous, even when starting with a weakly continuous $P$.

We finish this section with a generalization of \autoref{thm2new}. For $S=\mathbb R_+$, we say that $X$ is stochastically monotone if $P_x(X_1>y)$ is non-decreasing in $x$ for each $y\geq 0$.  

\begin{theorem} Let $X$ be an $\mathbb R_+$-valued stochastically monotone Markov chain for which $F(\cdot,y)$ is a continuous function for each $y\geq 0$, where $F(x,y)\overset{\Delta}{=}P_x(X_1\leq y)$. Suppose $X$ has a unique stationary distribution $\pi=(\pi(dx):x\geq 0)$ for which $\pi[a,b)>0$ for all $a,b$ for which $0\leq a<b<\infty$. Then, if $A_n=[0,n]$, $\Pi_n$ is non-empty for $n\geq 1$, and $\pi_n\in \Pi_n$, we have that
\begin{align*}
\pi_n\Rightarrow\pi
\end{align*}
as $n\rightarrow\infty$.
\end{theorem}
\begin{proof} We start by recognizing that we can apply the argument of \autoref{thm2new} to establish the existence of a strictly increasing sequence $(r(k):k\geq 1)$ converging to infinity such that
\begin{align*}
\sum_{k=0}^{\infty}r(k+1)\pi([k,k+1))<\infty.
\end{align*}
By setting $r(0)=0$ and defining $r(\cdot)$ between consecutive integers via linear interpolation, we construct a continuous function $r(\cdot)$ such that
\begin{align*}
\int_{\mathbb R_+}^{}r(x)\pi(dx)<\infty.\label{eq316}
\end{align*}
We next take advantage of the uniqueness of $\pi$ to guarantee that the shift operator defined by $\theta\circ X =(X_{1+n}:n\geq 0)$ is not only measure-preserving when $X_0$ has distribution $\pi$ but is also ergodic; see p. 141 of \citet{ashTopicsStochasticProcesses1975}. Let $T_0=0$ and $T_{i+1}=\inf\{n>T_i:X_n\leq 1\}$ for $i\geq 0$. The ergodic theorem for stationary sequences implies that
\begin{align}
\frac{1}{T_n}\sum_{j=0}^{T_n-1}r(X_j) \rightarrow \int_{\mathbb{R}_+}^{}r(x)\pi(dx)\qquad \text{a.s.}
\end{align}
and 
\begin{align*}
\frac{\sum_{j=0}^{n-1}\sum_{i=T_j}^{T_{j+1}-1}r(X_i)}{\sum_{j=0}^{n-1}(T_{j+1}-T_j)} \rightarrow \frac{\int_{[0,1]}^{}\pi(dx)E_x\sum_{i=0}^{T_1-1}r(X_i)}{\int_{[0,1]}^{}\pi(dx)E_xT_1}\qquad \text{a.s.}
\end{align*}
as $n\rightarrow\infty$, so that we may conclude that 
\begin{align*}
\int_{\mathbb{R}_+}^{}r(x)\pi(dx) = \frac{\int_{[0,1]}^{}\pi(dx)E_x\sum_{j=0}^{T_1-1}r(X_j)}{\int_{[0,1]}^{}\pi(dx)E_xT_1}.
\end{align*}
In view of \eqref{eq316}, we find that
\begin{align}
\int_{[0,1]}^{}\pi(dx)k(x)<\infty\label{eq317}
\end{align}
where
\begin{align*}
k(x)\overset{\Delta}{=}E_x\sum_{j=0}^{T_1-1}r(X_j).
\end{align*}
Let $\beta_n=\inf\{j\geq 0: X_j>n\}$ and observe that the ergodic theorem implies that $\beta_n<\infty$ a.s. when $X_0$ has distribution $\pi$. It is therefore evident that the set of $x$-values in $[0,1]$ for which $P_x(\beta_n<T_1)>0$ must have positive $\pi$-probability and the set of $x$-values for which $k(x)<\infty$ must have $\pi$-probability one, so that there exists $w\leq 1$ for which $P_w(\beta_n<T_1)>0$ and $k(w)<\infty$. Hence,
\begin{align*}
\int_{(n,\infty)}^{}P_w(X_{\beta_n}\in dy, \beta_n<T_1)k(y)\leq k(w)<\infty,
\end{align*}
so that there exists $y_n>n$ for which $k(y_n)<\infty$.

The stochastic monotonicity of $X$ and the monotonicity of $r$ imply that $k(\cdot)$ is a non-decreasing function, so that $k(z)<\infty$ for $z\leq y_n$, and hence $k(z)<\infty$ for all $z\in\mathbb \R_+$. Furthermore, the continuity of $r$ and weak continuity of $P$ ensure that $k$ is continuous on $(1,\infty)$; see \citet{glynnSolutionsPoissonEquation2022} for related arguments. 

Conditioning on $X_1$ establishes that for $x>1$,
\begin{align*}
k(x)=r(x)+\int_{(1,\infty)}^{}P_x(X_1\in dy)k(y).
\end{align*}
So, if we set $\tilde k(x)=k(x)$ for $x>1$ and $\tilde k(x)=k(1+)$ for $x\leq 1$, we find that
\begin{align*}
(P\tilde k)(x)&= k(1+) P_x(X_1\leq 1) + \int_{(1,\infty)}^{}P_x(X_1\in dy)k(y)\\
&\leq k(1+) + \int_{(1,\infty)}^{}P_{1+}(X_1\in dy) k(y)\\
&=k(1+) + k(1+) -r(1)\\
&=k(1+)+\tilde k(x) -r(1)
\end{align*}
for $x\in\mathbb [0,1]$. As a result, \autoref{a2} is satisfied with a continuous $\tilde k$, so that we can apply Theorems \ref{thm-2} and \ref{thm4}. Finally, observe that because $\tilde k$ is non-decreasing, it is evident that when $A_n=[0,n]$,
\begin{align*}
\sup_{x\in A_n}\tilde k(x) \leq \inf_{x\in A_n^c}\tilde k(x),
\end{align*}
so that $A_n$'s of this special form are legitimate truncation sets.
\end{proof}

This result generalizes \citet{gibsonMonotoneInfiniteStochastic1987} and \citet{tweedieTruncationApproximationsInvariant1998} to continuous state space. In particular, for stochastically monotone Markov chains, general augmentations yield convergent approximations when the truncation sets $A_n$ are chosen to be of the form $A_n=[0,n]$.

\section{Convergence for Strongly Uniformly Recurrent Markov Chains}\label{sec-4}
In this section, we use the theory of regeneration to establish convergence of general augmentation schemes for a class of general state space Markov chains that generalize the theory developed for Markov chains having Markov transition matrices. In contrast to the earlier sections, the theory developed here pertains to arbitrary truncation sequences, so that $(A_n:n\geq 1)$ can be any sequence for which $\emptyset\neq A_1 \subseteq A_2 \subseteq ...$ for which $\cup_{n=1}^\infty A_n=S$. 

We say that $X=(X_n:n\geq 0)$ is \emph{strongly uniformly recurrent} if there exists $\lambda>0$ and a probability $\phi$ such that
\begin{align}
P(x,dy) \geq \lambda \phi(dy)\label{eq41new}
\end{align}
for $x,y\in S$.
\begin{remark} Uniform recurrence requires the existence of $m\geq 1$ such that 
\begin{align}
P_x(X_m\in dy) \geq \lambda\phi(dy)\label{eq42new}
\end{align}
for $x,y\in S$, so that \eqref{eq41new} is clearly a strong version of \eqref{eq42new}. We note that \autoref{example-1} satisfies \eqref{eq42new} with $m=2$, and yet presents a setting in which general augmentation schemes can fail to converge. So, some condition (like strong uniform recurrence) is needed beyond uniform recurrence in order to guarantee convergence for general augmentations.
\end{remark}

Without any real loss of generality, we may assume that $\phi$ is supported on $A_1$ (at the possible cost of needing to reduce $\lambda$ in \eqref{eq41new}). As in our earlier sections, we allow a general augmentation $P_n=(P_n(x,dy):x,y\in A_n)$ associated with $A_n$ to take the form
\begin{align*}
P_n(x,dy) = P(x,dy) + \tilde R_n(x,dy)
\end{align*}
for $x,y\in A_n$, where $P_n$ is a stochastic kernel, and $\tilde R_n$ is a non-negative kernel.

In the presence of \eqref{eq41new}, \citet{athreyaNewApproachLimit1978} and \citet{nummelinSplittingTechniqueHarris1978} observed that the transition kernel $P$ can be put in the form
\begin{align}
P(x,dy) = \lambda\phi(dy) + (1-\lambda)H(x,dy)\label{eq43new}
\end{align}
for $x,y\in S$, where $H=(H(x,dy):x,y\in S)$ is a stochastic kernel. (Note that $H$ is defined so as to make \eqref{eq43new} valid.) With the ``splitting representation'' \eqref{eq43new} in hand, we can see that the right-hand side is a mixture of two distributions. Consequently, we have a probabilistic mechanism for how to envision the transitions of the Markov chain that evolves according to $P$. In particular, if $X_n=x$, we first flip a $\lambda$-coin having probability of ``heads'' given by $\lambda$. If the coin comes up heads, we distribute $X_{n+1}$ according to $\phi$, and $X$ regenerates at that time. Otherwise, we distribute $X_{n+1}$ according to $H(X_n, \cdot)$. If we let $P_\phi(\cdot)$ and $E_\phi(\cdot)$ denote the probability and expectation under which $X_0$ has distribution $\phi$, the theory of regenerative processes (see, for example, \citet{smithRegenerativeStochasticProcesses1955}) asserts that the unique stationary distribution $\pi$ associated with $P$ is given by
\begin{align}
\pi(\cdot) = \frac{E_\phi \sum_{j=0}^{\tau-1}I(X_j\in \cdot)}{E_\phi\tau}\label{eq44new}
\end{align}
where $\tau$ is the first time at which $X$ distributes itself according to $\phi$.

Because $\phi$ is supported on $A_1$, we note that
\begin{align}
P_n(x,dy) = \lambda\phi(dy) + q_n(x)H(x,dy) + r_n(x)R_n(x,dy)\label{eq45new}
\end{align}
where $\lambda+q_n(x)+r_n(x)=1$ for $x\in A_n$, $q_n(x),r_n(x)\geq 0$ and $R_n=(R_n(x,dy):x,y\in A_n)$ is a stochastic kernel. In view of \eqref{eq45new}, we can view the transitions of $X$ under $P_n$ as being implemented through a more complex randomization. In particular, if $X_m=x$, then with probability $\lambda$, $X_{m+1}$ distributes itself according to $\phi$, and $X$ regenerates. On the other hand, with probability $q_n(X_m)$, $X_{m+1}$ distributes itself according to $H(X_m,\cdot)$, while with probability $r_n(X_m)$, $X_{m+1}$ distributes itself according to $R_n(X_m,\cdot)$. As in the discussion of the dynamics of $X$ under the transition kernel $P$, we let $\tau$ be the first time at which $X$ regenerates and distributes itself according to $\phi$. Also, let $\beta_n-1$ be the first time at which $X_{m+1}$ is drawn from the distribution $R_n(X_m,\cdot)$, so that the conditional distribution of $\beta_n$ has probability mass function
\begin{align}
P(\beta_n=k+1|X_0,X_1,...,X_{k})&= \prod_{j=0}^{k-1}(1-r_n(X_j))r_n(X_k).\label{eq46new}
\end{align}
A final key observation is that 
\begin{align}
P_\phi((X_0,X_1,...,X_k)\in \cdot, \tau\land \beta_n>k) = P_\phi^n((X_0,...,X_k)\in\cdot, \tau\land\beta_n>k)\label{eq47new}
\end{align}
for $k\geq 0$, where $P_\phi^n(\cdot)$ is the probability under which $X_0$ has distribution $\phi$ and $X$ evolves under $P_n$.

We are now ready to state our main result.

\begin{theorem} Suppose that $X$ is strongly uniformly recurrent under $P$ with a unique stationary distribution $\pi$. Then, $X$ is strongly uniformly recurrent under $P_n$ with a unique stationary distribution $\pi_n$, and
\begin{align*}
\sup_{A\subseteq S}|\pi_n(A)-\pi(A)|\rightarrow 0
\end{align*}
as $n\rightarrow\infty$.
\end{theorem}

Note that $\pi_n$ converges to $\pi$ in total variation norm (rather than in the sense of weak convergence used elsewhere in this paper), regardless of how the $A_n$'s are chosen, and regardless of how the sequence of augmentations is defined.

\begin{proof} We first recognize that since $X$ regenerates at time $\tau$ under $P_\phi^n$, $\pi_n$ can be expressed in terms of regenerative cycle quantities (as in \eqref{eq44new}) as
\begin{align*}
\pi_n(\cdot) = \frac{E_\phi^n\sum_{j=0}^{\tau-1}I(X_j\in \cdot)}{E_\phi^n\tau},
\end{align*}
where $E_\phi^n(\cdot)$ is the expectation associated with $P_\phi^n$. We now couple the dynamics of $X$ under $P_\phi^n$ to its evolution under $P_\phi$ by drawing $X_m$ from $H(X_{m-1},\cdot)$ under $P_\phi$ whenever we draw $X_m$ from $R_n(X_{m-1},\cdot)$ under $P_\phi^n$. Then, on account of \eqref{eq47new},
\begin{align*}
|E_\phi^n&\sum_{j=0}^{\tau-1}I(X_j\in \cdot)-E_\phi\sum_{j=0}^{\tau-1}I(X_j\in \cdot)| \\
&= |\sum_{j=0}^{\infty}P_\phi^n(X_j\in\cdot, \beta_n\land \tau>j) + E_\phi^n\sum_{j=\beta_n}^{\tau-1}I(X_j\in \cdot, \tau>\beta_n)\\
&\qquad -\sum_{j=0}^{\infty}P_\phi(X_j\in \cdot, \beta_n\land \tau>j) - E_\phi \sum_{j=\beta_n}^{\tau-1}I(X_j\in\cdot, \tau>\beta_n)|\\
&\leq E_\phi^n(\tau-\beta_n)I(\tau>\beta_n) + E_\phi (\tau-\beta_n)I(\tau>\beta_n)\\
&\leq E_\phi^n\tau I(\tau>\beta_n) + E_\phi \tau I(\tau>\beta_n)\\
&\leq (E_\phi^n\tau^2)^{1/2} P_\phi^n(\tau>\beta_n)^{1/2} + (E_\phi \tau^2)^{1/2}P_\phi(\tau>\beta_n)^{1/2}.
\end{align*}
Since $\tau$ is geometric with parameter $\lambda$ under both $P_\phi$ and $P_\phi^n$, $E_\phi^n\tau^2 = E_\phi \tau^2\leq 2/\lambda^2$. On the other hand,
\begin{align*}
P_\phi^n(\tau>\beta_n) &= \sum_{j=1}^{\infty}P_\phi^n(\beta_n=j, \tau>j)\\
&=\sum_{j=1}^{\infty}E_\phi^n\prod_{k=0}^{j-2} q_n(X_k) r_n(X_{j-1})\\
&=\sum_{j=1}^{\infty}E_\phi\prod_{k=0}^{j-2} q_n(X_k) r_n(X_{j-1})\\
&= P_\phi(\tau>\beta_n).
\end{align*}
Since $q_n(X_k)\leq 1-\lambda$, it follows that
\begin{align*}
\prod_{k=0}^{j-1} q_n(X_k)r_n(X_k) \leq (1-\lambda)^j.
\end{align*}
But $r_n(X_j)\downarrow 0$ a.s. as $n\rightarrow\infty$, so the Dominated Convergence Theorem implies that $P_\phi(\tau>\beta_n)=P_\phi^n(\tau>\beta_n)\rightarrow 0$ as $n\rightarrow\infty$. Consequently, 
\begin{align*}
\sup_B |E_\phi^n \sum_{j=0}^{\tau-1}I(X_j\in B) -E_\phi \sum_{j=0}^{\tau-1}I(X_j\in B)|\rightarrow 0
\end{align*}
and this implies that $E_\phi^n\tau\rightarrow E_\phi \tau$ as $n\rightarrow\infty$, thereby proving the theorem.
\end{proof}

\section{Convergence of General Augmentations for Markov Jump Processes}\label{sec-5}
We now briefly describe the extension of our discrete time theory to the setting of Markov jump processes. In particular, suppose that $S$ is a finite or countably infinite state space. We say that $Q=(Q(x,y):x,y\in S)$ is a \emph{rate matrix} if $Q(x,y)\geq 0$ for $x\neq y$,
\begin{align*}
\lambda(x)\overset{\Delta}{=}\sum_{y\neq x}^{}Q(x,y)<\infty
\end{align*}
and
\begin{align*}
\lambda(x)=-Q(x,x)
\end{align*}
for $x\in S$. We shall assume that the associated Markov jump process $X$ is non-explosive; \citet{meynStabilityMarkovianProcesses1993} provide a Lyapunov condition that guarantees non-explosiveness. In the presence of non-explosiveness, $X=(X(t):t\geq 0)$ can be realized as a stochastic process having piecewise constant paths that are right continuous. \citet{millerStationarityEquationsContinuous1963} shows that when $X$ is non-explosive, then positive recurrence is equivalent to finding a stationary distribution $\pi=(\pi(x):x\in S)$ satisfying
\begin{align}
\pi Q=0,\label{eq41}
\end{align}
in which case $\pi$ is the stationary distribution of $X$.

For a given truncation $A_n\subseteq S$, we say that $Q_n$ is an \emph{augmentation} of $Q$ if $Q_n$ is a rate matrix on $A_n$ for which $Q_n(x,y)\geq Q(x,y)$ for all $x\neq y$ with $x,y\in A_n$.

\begin{Qtheorem}\label{a3new} There exist non-negative coercive functions $r=(r(x):x\in S)$ and $g=(g(x):x\in S)$, and $b<\infty$ for which
\begin{align*}
(Qg)(x)\leq -r(x)+b
\end{align*}
for $x\in S$.
\end{Qtheorem}
Assumption \autoref{a3new} is the continuous-time analog of \autoref{a1}. The continuous-time analog of \autoref{thm-1} is our final result.

\begin{theorem} \label{thm6} Suppose $X$ is an irreducible non-explosive Markov jump process with rate matrix $Q$ and probability $\pi$ satisfying \eqref{eq41}. If \autoref{a3new} is in force, then for any augmentation sequence $(Q_n:n\geq 1)$ for which $A_n=K_n(g)$, we have that
\begin{align*}
\pi_n\Rightarrow \pi
\end{align*}
as $n\rightarrow\infty$, provided that $\pi_n$ is a probability satisfying $\pi_nQ_n=0$.
\end{theorem}

The proof of \autoref{thm6} is essentially identical to that of \autoref{thm-1}. The key observation is that $\pi_nr\leq b$ for $n\geq 1$, (from \citet{glynnBoundingStationaryExpectations2008}) so that $(\pi_n:n\geq 1)$ is again tight. Continuous-time analogs to all our other main results can be similarly derived.

\bibliographystyle{apalike}
\bibliography{GeneralConvergence}

\end{document}

%% file: preamble.tex
\usepackage{amsmath,amsthm,amssymb,amsfonts,mathtools,enumerate,enumitem,setspace}
\usepackage[normalem]{ulem}
\usepackage[round]{natbib} 

\newcommand\norm[1]{\left\lVert#1\right\rVert}

\theoremstyle{plain}
\newtheorem{theorem}{Theorem}

\theoremstyle{definition}
\newtheorem{example}{Example}

\theoremstyle{remark}
\newtheorem{remark}{Remark}

\newtheoremstyle{mystyle}%
{0pt}%space above
{0pt}%space below
{}% body font
{}% paragraph indent amount
{\bfseries}% head font
{.}% punctuation after head
{.5em}% space after head
{}% headspec
\theoremstyle{mystyle}
\newtheorem{assumption}{\textbf{A}\ignorespaces}
\newenvironment{Qtheorem}[1][]
  {\quote\begin{assumption}[#1]}
  {\end{assumption}\ \endquote}

\usepackage{hyperref}
\hypersetup{
    colorlinks,
    citecolor=black,
    filecolor=red,
    linkcolor=black,
    urlcolor=red
}
\usepackage{indentfirst}
\numberwithin{equation}{section}
\allowdisplaybreaks

\newcommand{\R}{\mathbb R}